\documentclass[a4paper]{amsart}
\usepackage{amsfonts}
\usepackage{amssymb}
\usepackage{mathrsfs}
\usepackage{float}
\usepackage{amsmath}
\usepackage{amsfonts,amssymb,verbatim}
\usepackage{graphicx,color}
\usepackage{mathrsfs}
\usepackage{setspace}
\theoremstyle{plain}

\numberwithin{equation}{section}
\newtheorem{theorem}{Theorem}[section]
\newtheorem{proposition}[theorem]{Proposition}

\theoremstyle{remark}
\newtheorem{remark}[theorem]{\bf Remark}

\newcommand{\R}{{\mathbf R}}
\newcommand{\s}{{\mathcal S}}

\begin{document}

\title{The Fourier transform of multiradial functions}

\author{Fr\'ed\'eric Bernicot}
\address{CNRS - Universit\'e de Nantes, Laboratoire Jean Leray  2, rue de la Houssini\`ere,
44322 Nantes cedex 3, France }
\email{frederic.bernicot@univ-nantes.fr}

\author{Loukas Grafakos}
\address{Department of Mathematics, University of Missouri, Columbia, MO 65211, USA}
\email{grafakosl@missouri.edu}

\author{Yandan Zhang}
\address{Department of Mathematics, Zhejiang University of Science and Technology, Hangzhou,
Zhejiang, China}
\email{yandzhang@163.com}

\day=25 \month=05 \year=2013
\date{\today}

 \subjclass[2000]{42B10, 42B37}
  \keywords{Multiradial function, Fourier transform.}
 \thanks{The first author is partially supported by supported by the ANR under the project AFoMEN no. 2011-JS01-001-01. The second author was supported by grant DMS 0900946 of the National Science Foundation of the USA. The
third author was supported by the National Natural Science Foundation of China (Grant No. 11226108 and No. 11171306).}

\begin{abstract}
We obtain an exact formula for the
Fourier transform of multiradial functions, i.e., functions of the form
$\Phi(x)=\phi(|x_1|, \dots , |x_m|)$, $x_i\in \mathbf R^{n_i}$, in terms of the Fourier transform of the function $\phi$ on $\mathbf
R^{r_1}\times \cdots \times \mathbf R^{r_m}$,
where $r_i$ is either $1$ or $2$.
\end{abstract}

\maketitle

\section{Introduction}
~~~~~Let $m\geq1$, $n_1,\dots,n_m\geq1$ be integers.
Throughout this note, we will adhere to the following notation for the Fourier transform of a function $\Phi$ in $L^1(\mathbf{R}^{n_1+\cdots+n_m})$
$$
 F_{n_1,\dots,n_m}(\Phi)(\xi_1,\dots,\xi_m)
   =\int_{\mathbf{R}^{{n_m}}}\cdots\int_{\mathbf{R}^{{n_1}}}
\Phi(x_1,\dots,x_m)e^{-2\pi i(x_1\cdot\xi_1+\cdots+x_m\cdot\xi_m)}dx_1\cdots dx_m.
$$ 
The function $\Phi$ is called multiradial if there exists some  function $\phi$ on $(\mathbf{R}^+\cup\{0\})^m$ such that
\begin{equation}\label{1}
\Phi(x_1,\dots,x_m)=\phi(| x_1|,\dots,| x_m|)
\end{equation}
 for all $x_i\in \mathbf R^{n_i}$, where $|x_j|$ denotes the Euclidean norm of $x_j$.
  In the case $m=1$, $\Phi$ is simply called radial.
 Obviously, if $\Phi$ is multiradial, so is its Fourier transform, which only depends on $\phi$. Thus it is appropriate to use the notation
$$
\mathcal{F}_{n_1,\dots,n_m}(\phi)(r_1,\dots,r_m):=F_{n_1,\dots,n_m}(\Phi)(\xi_1,\dots,\xi_m),
$$
 where $r_1=|\xi_1|,\dots,r_m=|\xi_m|$, for the Fourier transform of a multiradial function $\Phi$ on $\mathbf R^{n_1+\cdots +n_m}$.

There exists an obvious identification between functions $\phi$ on $[0,\infty)^m$ and multi-even functions (functions that are even with respect to each of their variables) on $\mathbf{R}^m$  given by
$$
\phi_{ext}(t_1, \dots, t_m) = \phi (|t_1|,\dots , |t_m|)\, .
$$
Clearly, the restriction of $\phi_{ext}$ on $[0,\infty)^m$ is $\phi$.
We introduce the notation
$$
\widehat{\phi} := F_{1,\dots, 1}({  \phi_{ext} } ) \, .
$$
 Throughout this paper we denote the multi-even extension $\phi_{ext}$ of   $\phi$  also by $\phi$, and then  $\widehat{\phi} $
provides a shorter notation for $ F_{1,\dots, 1}({  \phi }) $,  which also coincides with $\mathcal F_{1,\dots, 1} (\phi) $ on $[0,\infty)^m$.

In the recent work of Grafakos and Teschl \cite{GT}  an explicit formula for the Fourier transform of a radial function $\Phi(x)=\phi(|x|)$ is given in terms of the one-dimensional Fourier transform of $\phi$ or the two-dimensional Fourier  transform of $(t,s)\mapsto \phi(|(t,s)|)$. In this work we extend this formula to
multiradial functions.  We obtain relatively straightforward formulas that relate  the Fourier transform on $\mathbf R^{m(k+2)}$  with that on $\mathbf R^{mk}$ but also new more complicated ones that relate the Fourier transform  on $\mathbf R^{m(k+1)}$ with that on $\mathbf R^{mk}$; the latter formulas   are valid only in the case of   compactly supported Fourier transforms, i.e., band-limited multiradial signals.

We have the following   results:

\begin{theorem} \label{thm1} Let $m\geq1$ and $k_i\in\mathbf{Z}^+ $ for $i=1,\dots , m$. Suppose that $\Phi$ is related to $\phi$ via \eqref{1} and that $\phi$ satisfies
$$ \int_{[0,\infty)^m} \prod_{j=1}^m (1+r_i)^{2k_j+1} |\phi(r_1,\dots, r_m)| dr <\infty.$$
Then   the following identities are valid:
\begin{align*}
& \mathcal{F}_{2k_1\!+\!1,\dots , 2k_m\!+\!1}(\phi)(r_1,\dots,r_m) \\
&=\frac{1}{(2\pi)^{k_1+\cdots +k_m}}\sum_{\ell_m=1}^{k_m}\frac{(-1)^{\ell_m}(2k_m-\ell_m-1)!}{2^{k_m-\ell_m}(k-\ell_m)!(\ell_m-1)!}
\frac{1}{r_m^{2k_m-\ell_m}} \notag\\
&\quad\cdots\sum_{\ell_1=1}^{k_1}\frac{(-1)^{\ell_1}(2k_1-\ell_1-1)!}{2^{k_1-\ell_1}(k_1-\ell_1)!(\ell_1-1)!}
\frac{1}{r_1^{2k_1-\ell_1}}\frac{\partial^{\ell_1+\cdots+\ell_m}\mathcal{F}_{1,\dots, 1}(\phi)}{\partial r_m^{\ell_m}\cdots\partial r_1^{\ell_1}} (r_1,\dots,r_m)
\end{align*}
and
\begin{align*}
&\mathcal{F}_{2k_1\!+\!2,\dots , 2k_m\!+\!2}(\phi)(r_1,\dots,r_m)\\
&=\frac{1}{(2\pi)^{k_1+\cdots +k_m}}\sum_{\ell_m=1}^{k_m}\frac{(-1)^{\ell_m}(2k_m-\ell_m-1)!}{2^{k_m-\ell_m}(k_m-\ell_m)!(\ell_m-1)!}
\frac{1}{r_m^{2k_m-\ell_m}} \notag\\
&\quad\cdots\sum_{\ell_1=1}^{k_1}\frac{(-1)^{\ell_1}(2{k_1}-\ell_1-1)!}{2^{k_1-\ell_1}(k_1-\ell_1)!(\ell_1-1)!}
\frac{1}{r_1^{2k_1-\ell_1}}\frac{\partial^{\ell_1+\cdots+\ell_m}\mathcal{F}_{2,\dots, 2}(\phi)}{\partial r_m^{\ell_m}\cdots\partial r_1^{\ell_1}}(r_1,\dots,r_m).
\end{align*}
\end{theorem}

\begin{remark} \label{rem} We prove the   identity
$$
 \mathcal{F}_{k_1+2,\dots ,k_m+2}(\phi)(r_1,\dots,r_m)=\frac{(-1)^m}{(2\pi)^m r_1\cdots r_m}\frac{\partial^m \mathcal{F}_{k_1,\dots ,k_m}(\phi)}{\partial r_m\cdots\partial r_1}(r_1,\dots,r_m)
 $$
for every   $k_i\in \mathbf Z\cup\{0\}$ and this can be iterated to give the claimed identities in Theorem \ref{thm1}.
\end{remark}

\begin{remark} The integrability assumption on $\phi$ allows us to consider the function $\Phi$ given by \eqref{1}, and defined on  $\mathbf{R}^{n}$ for any $n$ satisfying $1\le n\le   2(k_1+\cdots +  k_m +m)$. Then $\Phi\in L^1(\mathbf{R}^{n})$. \\
Using the fact the Fourier transform is a unitary operator on $L^2(\mathbf{R}^{n_1+\cdots+n_m})$ and by density, $L^1$-integrability of $\Phi$ in the above theorem can be replaced by $L^2$-integrability. About the associated recursion in Theorem 1.1 for the case of Schwartz functions, we refer the reader to \cite{LT,SP,SW} for related results. One could consider  analogous recursion
formulas for   multiradial distributions; this has been studied in the linear case in \cite{Sz,Z1,Z2}.
\end{remark}

\begin{remark} We have given formulas for the Fourier transform of
$\phi(|x_1|, \dots , |x_m|)$ when either all $x_i$ lie in odd-dimensional spaces or all
$x_i$ lie in even-dimensional spaces in terms of the Fourier transform on $\phi$ on $\mathbf R^m$ or $\mathbf R^{2m}$, respectively. Analogous formulas work for the
Fourier transform of
functions $\phi(|x_1|, \dots , |x_m|)$ where $x_i \in \mathbf R^{n_i}$ in terms of the
 Fourier transform of $\phi(t_1, \dots , t_m)$, where $t_i\in \mathbf R$ when $n_i$ is odd
 and $t_i\in \mathbf R^2$ when $n_i$ is even.
\end{remark}


\begin{theorem} (a) Let $\phi$ be an even function   on a real line whose Fourier transform $\widehat{\phi}$ is supported in the interval $[-A,A]$.  Suppose that  $\Phi$ is related to $\phi$ via \eqref{1} and that for    some $k\in \mathbf Z\cup\{0\}$ we have
$$ \int_{[0,\infty)} (1+r)^{2k+1} |\phi(r)| dr <\infty.$$

If $k=0$, then the following identity is valid:
\begin{equation}\label{TBSoo99}
\mathcal{F}_2(\phi)(r)
=2\int_{r}^{A}(\widehat{\phi}\, )^\prime(w)\frac{dw}{\sqrt{w^2-r^2}}\chi_{[0,A]}(r).
\end{equation}

When   $k\ge 1$ we have
\begin{align*}
\mathcal{F}_{2k+1}(\phi)(r)
=\frac{1}{(2\pi)^k}\sum_{\ell=1}^k\frac{(-1)^{\ell}(2k-\ell-1)!}{2^{k-\ell}(k-\ell)!(\ell-1)!}
\frac{1}{r^{2k-\ell}}\frac{d^{\ell} \widehat{\phi}}{d w^{\ell}}(r)\chi_{(0,A)}(r)
\end{align*}
and
\begin{align}\label{cc}
\mathcal{F}_{2k+2}(\phi)(r)
=\frac{2}{(2\pi)^{k}}\sum_{\ell=1}^{k}\frac{(-1)^{\ell}(2k-\ell-1)!}{2^{k-\ell}(k-\ell)!(\ell-1)!}\bigg(
\int_{r}^A\frac{1}{w^{2k-\ell}}\frac{d^{\ell+1}\, \widehat{\phi}}{d w^{\ell+1}}  (w)\frac{dw}{\sqrt{w^2-r^2}} \bigg)\chi_{(0,A)}(r)\, .
\end{align}


(b) Let   $m\geq2$ and let   $\phi$ be a function defined on $\mathbf R^m$ which is even with respect to  any variable.
Suppose   that the Fourier transform $\widehat{\phi}$  of  $\phi$ is supported in $[-A,A]^m$. Let $\Phi$ be related to $\phi$ via \eqref{1} and suppose that for some $k_j\in \mathbf Z\cup\{0\}$ we have
$$
\int_{[0,\infty)^m} \prod_{j=1}^m (1+r_j)^{2k_j+1} |\phi(r_1,\dots, r_m)| dr <\infty.
$$

When all $k_j=0$, then  we have
\begin{equation}\begin{split} \label{xcvfgtr}
&\mathcal{F}_{2,\dots , 2}(\phi)(r_1,\dots , r_m)  \\ &=2^m\int_{r_m}^{A}\cdots  \int_{r_1}^{A}\frac{\partial^m\widehat{\phi}}{\partial w_m\cdots \partial w_1}(w_1,\dots , w_m)
\frac{dw_1}{\sqrt{w_1^2-r_1^2}}\cdots \frac{dw_m}{\sqrt{w_m^2-r_m^2}} \, \chi_{(0,A)^m}(r_1,\dots,r_m).
\end{split}\end{equation}

If all $k_j\ge 1$ we have
\begin{align*}
&\mathcal{F}_{2k_1\!+\!1+\cdots + 2k_m\!+\!1}(\phi)(r_1,\dots,r_m)\\
&=\frac{1}{(2\pi)^{k_1+\cdots +k_m}}\sum_{\ell_1=1}^{k_1}\frac{(-1)^{\ell_1}(2k_1-\ell_1-1)!}{2^{k_1-\ell_1}(k_1-\ell_1)!(\ell_1-1)!}\cdots\sum_{\ell_m=1}^{k_m}\frac{(-1)^{\ell_m}(2k_m-\ell_m-1)!}{2^{k_m-\ell_m}(k_m-\ell_m)!(\ell_m-1)!}\notag\\
&\qquad\frac{1}{r_1^{2k_1-\ell_1}\cdots r_m^{2k_m-\ell_m}}\frac{\partial^{\ell_1+\cdots+\ell_m}\, \widehat{\phi}}
{\partial r_1^{\ell_1}\cdots\partial r_m^{\ell_m}}(r_1,\dots,r_m)\chi_{(0,A)^m}(r_1,\dots,r_m)
\end{align*}
and
\begin{align*}
&\mathcal{F}_{2k_1\!+\!2,\dots , 2k_m\!+\!2}(\phi)(r_1,\dots,r_m)\\
&=\frac{2^m}{(2\pi)^{k_1+\cdots +k_m}}\sum_{\ell_1=1}^{k_1}\frac{(-1)^{\ell_1}(2k_1-\ell_1-1)!}{2^{k_1-\ell_1}(k_1-\ell_1)!(\ell_1-1)!}\cdots\sum_{\ell_m=1}^{k_m}\frac{(-1)^{\ell_m}(2k_m-\ell_m-1)!}{2^{k_m-\ell_m}(k_m-\ell_m)!(\ell_m-1)!} \\
&\qquad \bigg(\int_{[r_1,A]}\cdots\int_{[r_m,A]}\frac{1}{w_1^{2k_1-\ell_1}\cdots w_m^{2k_m-\ell_m}}\frac{\partial^{\ell_1+\cdots+\ell_m+m}\, \widehat{\phi} }
{\partial w_1^{\ell_1+1}\cdots\partial w_m^{\ell_m+1}}(w_1,\dots,w_m) \\
&\qquad\qquad \cdots \frac{dw_1}{\sqrt{w_1^2-r_1^2}} \frac{dw_m}{\sqrt{w_m^2-r_m^2}} \bigg)\chi_{(0,A)^m}(r_1,\dots,r_m) \, .
\end{align*}
\end{theorem}



\begin{remark}
We conclude the following: Under the hypotheses of the preceding theorem (part (b)), if $\mathcal{F}_{1,\dots, 1}(\phi)$ has compact support, then so does $\mathcal{F}_{2,\dots, 2}(\phi)$. More generally, by combining these two theorems, we also deduce that for every integers $k_1,\dots, k_m$ then $\mathcal{F}_{k_1,\dots, k_m}(\phi)$ has compact support too. This property can also be obtained as a consequence of the finite speed of propagation of the Euclidean Laplace operator $\Delta_{\mathbf R^{n}}=\otimes_{j=1}^m \Delta_{\mathbf R^{k_i}}$, see \cite[Lemma 3.1]{BGSY}. Moreover, in the radial case this property can also be rephrased as follows: a Fourier band-limited function is also a Hankel band-limited function, for the ``$J_0$'' Hankel transform and refer the reader to \cite{C,R} for more details. The work of Rawn \cite{R} also provided an inspiration for identity \eqref{TBSoo99}. 
\end{remark}

\begin{remark} For $\Phi$ related to $\phi$ via \eqref{1}, under the hypotheses of the preceding theorem (part (b)), we have an exact formula for its Fourier transform, only in terms of the Fourier transform of the function $\phi$ on $\mathbf R^{1}\times \cdots \times \mathbf R^{1}$. 
\end{remark}

We will also give some examples in the last section and describe an application to the framework of bilinear Marcinkiewicz-type  Fourier multipliers. More precisely, we show that the transformation consisting to replace a bi-even bilinear kernel $K$ on $\R$ by a bilinear kernel $\widetilde{K}$ on $\R^n$ with $\widetilde{K}(y,z)= (|y||z|)^{-n+1} K(|y|,|z|)$ preserves the Marcinkiewicz conditions (see Subsection \ref{subsec:m} for details).

\section{Proofs}

\begin{proof}[Proof of Theorem 1.1]

For  simplicity of exposition, we only consider the case where $k_1=\cdots = k_m=n$. The general case  only presents notational  differences. Throughout the proof we denote
by $J_\nu$ the Bessel function of order $\nu$ and by
$\widetilde{J}_\nu(t)=t^{-\nu}J_\nu(t)$.

Using polar coordinates, the Fourier transform of an integrable radial function $\Phi$ on $\mathbf{R}^{mn}$ is given by
\begin{align*}
&F_{n,\dots ,n}(\Phi)(\xi_1,\xi_2,\dots,\xi_m)\\
&=\int_0^\infty\!\!\!\cdots\!\! \int_0^\infty\phi(s_1,\dots,s_m)\int_{(S^{n-1})^m} e^{-2\pi i s\xi\cdot\theta}\, d\theta s_1\cdots s_mds_1\cdots ds_m \notag\\
&=(2\pi)^m\int_0^\infty\cdots\int_0^\infty\phi(s_1,\dots,s_m)J_{\frac{n}{2}-1}(2\pi s_1|\xi_1|)
\left(\frac{s_1}{|\xi_1|}\right)^{\frac{n}{2}-1}s_1ds_1\notag\\
&\qquad \cdots J_{\frac{n}{2}-1}(2\pi s_m|\xi_m|)\left(\frac{s_m}{|\xi_m|}\right)^{\frac{n}{2}-1}s_mds_m \notag\\
&=(2\pi)^{\frac{mn}{2}}\int_{[0,\infty]^m} \phi(s_1,\dots,s_m)\widetilde{J}_{\frac{n}{2}-1}(2\pi s_1r_1)
s_1^{n}\frac{ds_1}{s_1}\cdots\widetilde{J}_{\frac{n}{2}-1}(2\pi s_mr_m)s_m^{n}\frac{ds_m}{s_m} \notag\\
&:=\mathcal{F}_{n,\dots ,n}(\phi)(r_1,\dots,r_m),
\end{align*}
where    $|\xi_1|=r_1,\dots, |\xi_m|=r_m.$

A useful fact that will be used is
  that $\{-\frac{1}{2\pi}\frac{1}{r_i}\frac{\partial}{\partial r_i}\}_{i=1}^m $ commute for different values of $i$.

We differentiate $\mathcal{F}_{n,\dots ,n}(\phi)(r_1,\dots,r_m)$ with respect with $r_1$.
Using the identity
$$
\frac{d}{d t}\widetilde{J}_\nu(t)=-t\widetilde{J}_{\nu+1}(t),
$$ which holds for
all $t>0$, we obtain
\begin{align*}
\frac{\partial}{\partial r_1} \mathcal{F}_{n,\dots ,n}(\phi) (r_1,\dots,r_m) =-(2\pi)^{\frac{mn}{2}+2}r_1\int_0^\infty\cdots\int_0^\infty\phi(s_1,\dots,s_m) \hspace{1.5cm}  \\
 \widetilde{J}_{\frac{n+2}{2}-1}(2\pi s_1r_1)s_1^{n+2-1}ds_1 \cdots\widetilde{J}_{\frac{n}{2}-1}(2\pi s_mr_m)s_m^{n-1}ds_m.
\end{align*}

Differentiating with respect to the remaining variables $r_2, \dots , r_m$ we obtain
\begin{align*}
& \!\!\!\!\!
\frac{\partial^m}{\partial r_m\cdots\partial r_1}(\mathcal{F}_{n,\cdots ,n}(\phi))(r_1,\dots,r_m)\\
& =
(-1)^m(2\pi)^{2m}(2\pi)^{\frac{mn}{2}}r_1\cdots r_m\int_0^\infty\int_0^\infty\phi(s_1,\dots,s_m)\notag\\
&\qquad\qquad \widetilde{J}_{\frac{n+2}{2}-1}(2\pi s_1r_1)s_1^{n+2-1}ds_1
\cdots\widetilde{J}_{\frac{n+2}{2}-1}(2\pi s_mr_m)s_m^{n+2-1}ds_m\notag\\
&=(-1)^m(2\pi)^mr_1\cdots r_m\mathcal{F}_{n+2,\dots ,n+2}(\phi)(r_1,\dots,r_m)
\end{align*}
or
\begin{align}\label{88}
\mathcal{F}_{n+2,\dots ,n+2}(\phi)(r_1,\dots,r_m)&=(-1)^m\frac{1}{(2\pi)^mr_1\cdots r_m}\frac{\partial^m \mathcal{F}_{n,\cdots ,n}(\phi)}{\partial r_m\cdots\partial r_1}(r_1,\dots,r_m)\notag\\
&=\Big(-\frac{1}{2\pi}\frac{1}{r_m}\frac{\partial}{\partial r_m}\Big)\cdots
\Big(-\frac{1}{2\pi}\frac{1}{r_1}\frac{\partial}{\partial r_1}\Big) \mathcal{F}_{n,\dots ,n}(\phi ) (r_1,\dots,r_m).
\end{align}

It is easy to check the interchanging differentiation and integration in the preceding calculations is permissible because of the hypothesis on the integrability of $\Phi$ which translates to a condition about the integrability of $\phi(s_1,\dots, s_m)(s_1^2+\cdots+s_m^2)^{n-1}$ for all $n\le 2(mk+m)$.

For $k\in (\mathbf{Z}^+)^m$, using \eqref{88} by induction on $n$,
starting with $n=1$, we obtain
\begin{align*}
&\!\!\!\!\! \mathcal{F}_{ 2k_1\!+\! 1,\dots ,2k_m\!+\! 1}(\phi)(r_1,\dots,r_m)\\
&=\Big(-\frac{1}{2\pi}\frac{1}{r_m}\frac{\partial}{\partial r_m}\Big)^{k_m}\cdots
\Big(-\frac{1}{2\pi}\frac{1}{r_1}\frac{\partial}{\partial r_1}\Big)^{k_1}
(\mathcal{F}_{1,\dots ,1}\emph{}(\phi))(r_1,\dots,r_m)\notag\\
&=\Big(-\frac{1}{2\pi}\frac{1}{r_m}\frac{\partial}{\partial r_m}\Big)^{k_m}
\cdots\Big(-\frac{1}{2\pi}\frac{1}{r_2}\frac{\partial}{\partial r_2}\Big)^{k_2}
 \\
&\qquad\qquad\sum_{\ell_1=1}^{k_1} \frac{(-1)^{\ell_1}(2k_1-\ell_1-1)!}{2^{k_1-\ell_1}(k_1-\ell_1)!(\ell_1-1)!}
\frac{1}{r_1^{2k_1-\ell_1}}\frac{\partial^{\ell_1}\mathcal{F}_{1,\dots,1}(\phi)}{\partial r_1^{\ell_1}} (r_1,\dots,r_m)\notag\\
&=\frac{1}{(2\pi)^{k_1+\cdots +k_m}}\sum_{\ell_m=1}^{k_m}\frac{(-1)^{\ell_m}(2k_m-\ell_m-1)!}{2^{k_m-\ell_m}(k_m-\ell_m)!(\ell_m-1)!}
\frac{1}{r_m^{2k_m-\ell_m}}\notag\\
&\qquad\qquad\cdots\sum_{\ell_1=1}^{k_1}\frac{(-1)^{\ell_1}(2k_1-\ell_1-1)!}{2^{k_1-\ell_1}(k_1-\ell_1)!(\ell_1-1)!}
\frac{1}{r_1^{2k_1-\ell_1}}\frac{\partial^{\ell_1+\cdots+\ell_m}\mathcal{F}_{1,\dots,1}(\phi)}{\partial r_m^{\ell_m}\cdots\partial r_1^{\ell_1}}(r_1,\dots,r_m)\notag\\
\end{align*}
and likewise we obtain
\begin{align*}
&\!\!\!\!\mathcal{F}_{2k_1\!+\!2,\dots ,2k_1\!+\!2}(\phi)(r_1,\dots,r_m) \\
&=\frac{1}{(2\pi)^{k_1+\cdots +k_m}}\sum_{\ell_m=1}^{k_1}\frac{(-1)^{\ell_m}(2k_m-\ell_m-1)!}{2^{k_m-\ell_m}(k_m-\ell_m)!(\ell_m-1)!}
\frac{1}{r_m^{2k-\ell_m}}\notag\\
&\qquad\qquad\cdots\sum_{\ell_1=1}^{k_1}\frac{(-1)^{\ell_1}(2k_1-\ell_1-1)!}{2^{k_1-\ell_1}(k_1-\ell_1)!(\ell_1-1)!}
\frac{1}{r_1^{2k_1-\ell_1}}\frac{\partial^{\ell_1+\cdots+\ell_m}\mathcal{F}_{2,\dots,2}(\phi) }{\partial r_m^{\ell_m} \cdots\partial r_1^{\ell_1}} (r_1,\dots,r_m).
\end{align*}

This completes the proof of Theorem 1.1.
\end{proof}

\bigskip

\begin{proof}[Proof of Theorem 1.5]

 We prove this theorem with $A=\pi.$ If this case is proved, then we can take $\phi_0(t)=\frac{\pi}{A}\phi(\frac{\pi}{A}t)$ and by a change of variables we obtain \eqref{TBSoo99} and \eqref{cc} in Theorem 1.5.

\smallskip
\noindent {\bf Step 1.}
It is a well known fact (see \cite{G1}) that
\begin{align}
F_2(\Phi)(\xi) =2\pi\int_0^\infty\phi(s)J_0(2\pi s| \xi |)sds
=\mathcal{F}_2(\phi)(r)\, ,
\end{align}
where $J_0(t)=\frac{1}{\pi}\int_{-1}^{1}e^{ist}\frac{ds}{\sqrt{1-s^2}}$ is the Bessel function
of order zero.

In this step, we want to prove that given $\phi$ even function on the real line,
there exists one and only one function   $f$ on a real line such that
\begin{align}\label{15}
\phi(x)=\int_0^\pi f(u)J_0(2\pi ux)u\, du.
\end{align}

First, we look for necessary conditions on $f$, to be a solution of \eqref{15}.
So momentarily assume that such an $f$ exists,
by applying a change of variables and Fubini's theorem, we obtain
\begin{align}
\int_0^\pi f(u)J_0(2\pi ux)u\,du&
=\frac{1}{\pi}\int_0^\pi f(u)u\int_{-1}^{1}e^{i2\pi uxs}\frac{ds}{\sqrt{1-s^2}}\, du\notag\\
&=\frac{1}{\pi}\int_0^\pi f(u)u\int_{-u}^{u}e^{i2\pi wx}\frac{dw}{\sqrt{u^2-w^2}}d u\notag\\
&=\int_{-\pi}^\pi e^{2\pi iwx}\left\{\frac{1}{\pi}\int_{\mid w\mid}^{\pi}f(u)u\frac{du}{\sqrt{u^2-w^2}}\right\}dw. \label{eq:6}
\end{align}

Thus, we rewrite \eqref{15} as
\begin{align}
\phi(x)=\int_{-\pi}^\pi e^{2\pi iwx}\left\{\frac{1}{\pi}\int_{\mid w\mid}^{\pi}f(u)\frac{udu}{\sqrt{u^2-w^2}}
\right\}dw.
\end{align}

On the other hand, recalling that  $\widehat{\phi}$ is supported in $[-\pi,\pi]$, we have $\phi(x)=\int_{-\pi}^\pi\widehat{\phi}(w)e^{2\pi iwx}dw$ and thus by identifying with \eqref{eq:6}, it comes
\begin{align}\label{18}
\widehat{\phi}(w)=\frac{1}{\pi}\int_{\mid w\mid}^{\pi}f(u)\frac{udu}{\sqrt{u^2-w^2}}.
\end{align}

Since $\phi$ is even, so is $\widehat{\phi}$, thus it is sufficient to deal with the case $w>0$.

Integrating both sides of \eqref{18} with respect to $\frac{wdw}{\sqrt{w^2-y^2}}$ we obtain
\begin{align}\label{19}
h(y):=\int_{y}^{\pi}\widehat{\phi}(w)\frac{wdw}{\sqrt{w^2-y^2}}
=\frac{1}{\pi}\int_{y}^{\pi}\int_{w}^{\pi}f(u)\frac{udu}{\sqrt{u^2-w^2}}\frac{wdw}{\sqrt{w^2-y^2}}.
\end{align}

But an easy change of variables shows that
 $\int_{y}^{u}\frac{wdw}{\sqrt{w^2-y^2}{\sqrt{u^2-w^2}}}=\frac{\pi}{2}.$ Then applying Fubini's theorem, we deduce
\begin{align}\label{20}
h(y) =\frac{1}{\pi}\int_{y}^{\pi}f(u)u\int_{y}^{u}\frac{wdw}{\sqrt{u^2-w^2}\sqrt{w^2-y^2}}du
 =\frac{1}{2}\int_{y}^{\pi}f(u)udu.
\end{align}

Combining \eqref{19} with \eqref{20}, we get
\begin{align}\label{21}
\int_{y}^{\pi}f(u)udu=2\int_{y}^{\pi}\widehat{\phi}(w)\frac{wdw}{\sqrt{w^2-y^2}}.
\end{align}

We integrate by parts in  \eqref{21}, recalling the support of $\widehat{\phi}$, and differentiating with respect to $y$ we obtain
\begin{align*}
-f(y)y
&=2\frac{d}{dy}\left(\sqrt{\pi^2-y^2}\widehat{\phi}(\pi)-\int_{y}^{\pi}\sqrt{w^2-y^2}(\widehat{\phi}\, )^\prime(w)dw\right)\notag\\
&=-2\int_{y}^{\pi}\frac{y}{\sqrt{w^2-y^2}}(\widehat{\phi}\, )^\prime(w)dw
\end{align*}
thus
\begin{align}\label{100}
f(y)=2\int_{y}^{\pi}(\widehat{\phi}\, )^\prime(w)\frac{dw}{\sqrt{w^2-y^2}}.
\end{align}
Once this calculation is done, it is quite easy to check that the function $f$ given in
\eqref{100} satisfies \eqref{15} by reversing the preceding steps.
Moreover, the previous computations yield that this solution of \eqref{15} is the only one.

\smallskip

\noindent {\bf Step 2.}
For   functions $\phi$ such that $\int_0^\infty|\phi(s)|s\, ds<\infty$ we define an operator
$$
U(\phi)(r)=\int_0^\infty\phi(s)J_0(2\pi sr)\, sds .
$$
We want to prove the identity
\begin{equation}\label{200}
U^2(\phi)(t)=\frac{1}{2\pi}\phi(t).
\end{equation}
To prove \eqref{200}, it is enough to show that for all $t>0$ we have
\begin{align}\label{300}
\int_0^\infty\int_0^\infty\phi(s)J_0(2\pi sr)sdsJ_0(2\pi rt)rdr=\frac{1}{2\pi}\phi(t).
\end{align}

We start with the identity (see \cite{W} page 406)
\begin{equation}\label{400}
t\int_0^\infty J_1(2\pi tr)J_0(2\pi sr)dr = \begin{cases} 1 & \qquad\text{$s<t$,} \\
0 & \qquad\text{$s>t$.} \end{cases}
\end{equation}
Multiplying \eqref{400} by $\phi(s)s$ and integrating from 0 to $\infty$, we obtain
\begin{align}\label{500}
\int_0^\infty \phi(s)st\int_0^\infty J_1(2\pi tr)J_0(2\pi sr)drds=\int_0^t\phi(s)sds.
\end{align}

Using that $\frac{d}{du}(u^\nu J_\nu(u))=u^\nu J_{\nu-1}(u)$, and differentiating both sides of \eqref{500} with respect to $ t$, we get
\begin{align*}
\int_0^\infty \phi(s)s\int_0^\infty 2\pi trJ_0(2\pi tr)J_0(2\pi sr)drds=\phi(t)t.
\end{align*}

This proves \eqref{300} and hence \eqref{200}.

\noindent {\bf  Step 3.}
In view of the result of Step 1, there exists a function $f$ such that
\begin{align}\label{700}
\mathcal{F}_2(\phi)(r)
&=2\pi\int_0^\infty\phi(s)J_0(2\pi sr)sds \notag\\
&=2\pi\int_0^\infty\int_0^\infty f(u)\chi_{[0,\pi]}(u)J_0(2\pi su)uduJ_0(2\pi sr)sds \notag\\
&= f(r)\chi_{[0,\pi]}(r) \notag\\
&=2\int_{r}^{\pi}(\widehat{\phi}\, )^\prime(w)\frac{dw}{\sqrt{w^2-r^2}}\chi_{[0,\pi]}(r).
\end{align}
which proves \eqref{TBSoo99}.

Combining \eqref{700} with the result of Theorem 1.1. when $m=1$, we obtain
\begin{align}\label{800}
\mathcal{F}_4(\phi)(r)
&=-\frac{1}{2\pi}\frac{1}{r}\frac{d}{dr}(\mathcal{F}_2(\phi))(r)\notag\\
&=-2\frac{1}{2\pi}\frac{1}{r}\frac{d}{dr}\left(-\int_{r}^{\pi}\frac{d}{dw}\Big(\frac{(\widehat{\phi}\, )^\prime(w)}{w}\Big)\sqrt{w^2-r^2}dw\right)\chi_{(0,\pi)}(r)\notag\\
&=\frac{2}{2\pi}\bigg( \int_{r}^{\pi}\frac{d}{dw}\Big(\frac{(\widehat{\phi}\, )^\prime(w)}{w}\Big)\frac{dw}{\sqrt{w^2-r^2}}\bigg)\chi_{(0,\pi)}(r).
\end{align}

Differentiating \eqref{800} $k-1$ times, we obtain \eqref{cc} with $A=\pi$.
Due to symmetry of $\phi$, the other formula in Theorem 1.5 is directly deduced from the first equation in Theorem 1.1.

\smallskip

We now proceed to part (b).
For simplicity we look at the case where $m=2$ and $A=\pi$.
\smallskip

\noindent {\bf Step 1.}
For $\Phi$ on $\mathbf{R}^4$ and $\xi\in\mathbf{R}^2$, $\eta\in\mathbf{R}^2$
\begin{align*}
F_{2,2}(\Phi)(\xi,\eta)&=
\int_0^\infty\int_0^\infty\phi(s_1,s_2)\int_{S^1}\int_{S^1}e^{-2\pi s_1\eta\cdot
\theta_1}e^{-2\pi s_2\xi\cdot\theta_2}d\theta_1d\theta_2s_1s_2ds_1ds_2 \\
&=(2\pi)^2\int_0^\infty\int_0^\infty\phi(s_1,s_2)J_0(2\pi s_1|\xi|)s_1ds_1J_0(2\pi s_2|\eta|)s_2ds_2 \\
&:=\mathcal{F}_{2,2}(\phi)(r_1,r_2),
\end{align*}
where $\Phi(\xi,\eta)=\phi(|\xi|,|\eta|)$, $J_0(t)=\frac{1}{\pi}\int_{-1}^{1}e^{ist}\frac{ds}{\sqrt{1-s^2}}$ and $|\xi|=r_1$, $|\eta|=r_2.$

We proceed as for the part (a). So we first aim to show that there exists a unique function $f$ on $[0,\pi]^2$ such that
\begin{align}\label{kk}
\phi(x_1,x_2)=\int_0^\pi\int_0^\pi f(u_1,u_2)J_0(2\pi u_1x_1)J_0(2\pi u_2x_2)u_1u_2du_1du_2.
\end{align}

Assume momentarily  that such a function exists. For a
function $h$ we have
\begin{align}
\int_0^\pi h(u)J_0(2\pi ux)udu&=\frac{1}{\pi}\int_0^\pi h(u)u\int_{-1}^{1}e^{2\pi i uxs}\frac{ds}{\sqrt{1-s^2}}du\notag\\
&=\frac{1}{\pi}\int_0^\pi h(u)u\int_{-u}^{u}e^{2\pi iwx}\frac{dw}{\sqrt{u^2-w^2}}d u\notag\\
&=\int_{-\pi}^\pi e^{2\pi iwx}\left\{\frac{1}{\pi}\int_{\mid w\mid}^{\pi}h(u)u\frac{du}{\sqrt{u^2-w^2}}\right\}dw.
\end{align}
Thus, we rewrite \eqref{kk} as
\begin{align*}
\lefteqn{\phi(x_1,x_2)=} & & \\
 & & \frac{1}{\pi^2}\int_{-\pi}^\pi\int_{-\pi}^\pi e^{2\pi iw_1x_1}e^{2\pi iw_2x_2}\left\{\int_{\mid
w_2\mid}^{\pi}\int_{\mid w_1\mid}^{\pi}f(u_1,u_2)\frac{u_1du_1}{\sqrt{u_1^2-w_1^2}}
\frac{u_2du_2}{\sqrt{u_2^2-w_2^2}}\right\}dw_1dw_2.
\end{align*}
Recalling the support of $\widehat{\phi}$, we have $\phi(x_1,x_2)=\int_{-\pi}^\pi\int_{-\pi}^\pi\widehat{\phi}(w_1,w_2)e^{2\pi i (w_1x_1+w_2x_2)}dw_1dw_2$. Thus the function $f$ on $\mathbf{R}^2$ would satisfy:
\begin{align}\label{gg}
\widehat{\phi}(w_1,w_2)=\frac{1}{\pi^2}\int_{\mid w_2\mid}^{\pi}\int_{\mid w_1\mid}^{\pi}f(u_1,u_2)\frac{u_1du_1}{\sqrt{u_1^2-w_1^2}}
\frac{u_2du_2}{\sqrt{u_2^2-w_2^2}}.
\end{align}

Since $\phi$ is even, it is sufficient to consider the case $w_1,w_2>0$.

Then integrating both sides of \eqref{gg} with respect to
$\frac{w_2dw_2}{\sqrt{w_2^2-y_2^2}}\frac{w_1dw_1}{\sqrt{w_1^2-y_1^2}}$
we obtain
\begin{align}
h(y_1,y_2):&=\int_{y_1}^{\pi}\int_{y_2}^{\pi}\widehat{\phi}(w_1,w_2)\frac{w_2dw_2}{\sqrt{w_2^2-y_2^2}}\frac{w_1dw_1}{\sqrt{w_1^2-y_1^2}}\notag\\
&=\frac{1}{\pi^2}\int_{y_1}^{\pi}\int_{y_2}^{\pi}\int_{w_2}^{\pi}\int_{w_1}^{\pi}f(u_1,u_2)
\frac{u_1du_1}{\sqrt{u_1^2-w_1^2}}\frac{u_2du_2}{\sqrt{u_2^2-w_2^2}}\frac{w_2dw_2}{\sqrt{w_2^2-y_2^2}}\frac{w_1dw_1}{\sqrt{w_1^2-y_1^2}}.
\end{align}

Note that $\int_{y}^{u}\frac{wdw}{\sqrt{w^2-y^2}{\sqrt{u^2-w^2}}}=\frac{\pi}{2}.$ Applying Fubini's theorem three times, we get
\begin{align}\begin{split}\label{ggg}
h(y_1,y_2)&=\frac{1}{\pi^2}\int_{y_1}^{\pi}\int_{y_2}^{\pi}\left\{\int_{w_1}^{\pi}f(u_1,u_2)
\frac{u_1du_1}{\sqrt{u_1^2-w_1^2}}\right\}\int_{y_2}^{u_2}\frac{w_2dw_2}{\sqrt{w_2^2-y_2^2}{\sqrt{u_2^2-w_2^2}}}u_2du_2\frac{w_1dw_1}{\sqrt{w_1^2-y_1^2}} \\
&=\frac{1}{2\pi}\int_{y_1}^{\pi}\int_{y_2}^{\pi}\left\{\int_{w_1}^{\pi}f(u_1,u_2)
\frac{u_1du_1}{\sqrt{u_1^2-w_1^2}}\right\}u_2du_2\frac{w_1dw_1}{\sqrt{w_1^2-y_1^2}} \\
&=\frac{1}{2\pi}\int_{y_2}^{\pi}\left\{\int_{y_1}^{\pi}\int_{w_1}^{\pi}f(u_1,u_2)
\frac{u_1du_1}{\sqrt{u_1^2-w_1^2}}\frac{w_1dw_1}{\sqrt{w_1^2-y_1^2}}\right\}u_2du_2 \\
&=\frac{1}{4}\int_{y_2}^{\pi}\int_{y_1}^{\pi}f(u_1,u_2)u_1du_1u_2du_2.
\end{split}\end{align}

Using \eqref{gg} and  and \eqref{ggg}, we deduce
\begin{align}\label{bb}
\int_{y_2}^{\pi}\int_{y_1}^{\pi}f(u_1,u_2)u_1du_1u_2du_2
=4\int_{y_1}^{\pi}\int_{y_2}^{\pi}\widehat{\phi}(w_1,w_2)\frac{w_2dw_2}{\sqrt{w_2^2-y_2^2}}\frac{w_1dw_1}{\sqrt{w_1^2-y_1^2}}.
\end{align}

We can recover $f$ from this equation.
Differentiating \eqref{bb} with respect with $y_1$ and $y_2$, we obtain
\begin{align*}
&f(y_1,y_2)y_1y_2 \\
&=4\frac{\partial^2}{\partial y_2\partial y_1}
\left(\int_{y_1}^{\pi}\left\{\int_{y_2}^{\pi}\widehat{\phi}(w_1,w_2)\frac{w_2dw_2}{\sqrt{w_2^2-y_2^2}}\right\}\frac{w_1dw_1}{\sqrt{w_1^2-y_1^2}}\right) \\
&=4\frac{\partial^2}{\partial y_2\partial y_1} \\
&\left(\sqrt{\pi^2-y_1^2}\int_{y_2}^{\pi}\widehat{\phi}(\pi,w_2)\frac{w_2dw_2}{\sqrt{w_2^2-y_2^2}}
-\int_{y_1}^{\pi}\sqrt{w_1^2-y_1^2}\left\{\int_{y_2}^{\pi}\frac{\partial  \widehat{\phi} }
{\partial w_1}(w_1,w_2)\frac{w_2dw_2}{\sqrt{w_2^2-y_2^2}}\right\}dw_1\right).
\end{align*}

Recalling the support of $\widehat{\phi}$, we get
\begin{align*}
& f(y_1,y_2)y_1y_2\\
&=4\frac{\partial^2}{\partial y_2\partial y_1} \\
&\left(-\int_{y_1}^{\pi}\sqrt{w_1^2-y_1^2}\left\{\sqrt{\pi^2-y_2^2}\frac{\partial  \widehat{\phi}  }{\partial w_1}(\pi,w_2)
-\int_{y_2}^{\pi}\sqrt{w_2^2-y_2^2}\frac{\partial^2 \widehat{\phi}  }{\partial w_2\partial w_1}
(w_1,w_2)dw_2\right\}dw_1\right) \\
&=4\frac{\partial^2}{\partial y_2\partial y_1}
\left(\int_{y_1}^{\pi}\sqrt{w_1^2-y_1^2}\int_{y_2}^{\pi}\sqrt{w_2^2-y_2^2}\frac{\partial^2 \widehat{\phi}  }{\partial w_2\partial w_1}
(w_1,w_2)dw_2dw_1\right) \\
&=4\int_{y_1}^{\pi}\frac{y_1}{\sqrt{w_1^2-y_1^2}}\int_{y_2}^{\pi}\frac{y_2}{\sqrt{w_2^2-y_2^2}}
\frac{\partial^2 \widehat{\phi}  }{\partial w_2\partial w_1}(w_1,w_2)dw_2dw_1
\end{align*}
or
\begin{align*}
f(y_1,y_2)
&=4\int_{y_1}^{\pi}\int_{y_2}^{\pi}\frac{\partial^2 \widehat{\phi}}{\partial w_2\partial w_1}  (w_1,w_2)
\frac{dw_2}{\sqrt{w_2^2-y_2^2}}\frac{dw_1}{\sqrt{w_1^2-y_1^2}}.
\end{align*}
We notice that this function $f$ we have constructed in this way satisfies \eqref{kk} by reversing the preceding
steps and is the unique solution.
\smallskip

\noindent {\bf Step 2.}
For  functions $\phi$ on $\mathbf{R}^2$ such that $\int_0^\infty\int_0^\infty|\phi(s_1,s_2)|s_1s_2ds<\infty$, we
 define an operator $U$ by setting
$$
U(\phi)(r_1,r_2)=\int_0^\infty\int_0^\infty\phi(s_1,s_2)J_0(2\pi s_1r_1)s_1ds_1J_0(2\pi s_2r_2)s_2ds_2.
$$
We want to prove the following identity
\begin{equation}\label{TBP}
U^2(\phi)(t_1,t_2)=\frac{1}{(2\pi)^2}\phi(t_1,t_2).
\end{equation}
It is enough to show
\begin{align*}
& \int_0^\infty\int_0^\infty\int_0^\infty\int_0^\infty\phi(s_1,s_2)J_0(2\pi s_1r_1)s_1ds_1J_0(2\pi s_2r_2)s_2ds_2J_0(2\pi r_1t_1)r_1dr_1J_0(2\pi r_2t_2)r_2dr_2\\
& =\frac{1}{(2\pi)^2}\phi(t_1,t_2).
\end{align*}

We make use of the fact   below that  can be found in \cite{W} page 406:
$$
t_2t_1\int_0^\infty\int_0^\infty J_1(2\pi t_1r_1)J_0(2\pi s_1r_1)dr_1J_1(2\pi t_2r_2)J_0(2\pi s_2r_2)dr_2
=\begin{cases} 1\qquad \textup{if $s_1<t_1$ and $s_2<t_2$.} \\
0 \qquad \textup{otherwise.} \end{cases}
$$
Multiplying the preceding identity by $\phi(s_1,s_2)s_1s_2$, integrating both sides in $s_1$ and $s_2$, we obtain
\begin{align}\label{abc}
\int_0^\infty\int_0^\infty\phi(s_1,s_2)s_1s_2t_2t_1\int_0^\infty\int_0^\infty J_1(2\pi t_1r_1)J_0(2\pi s_1r_1)dr_1J_1(2\pi t_2r_2)J_0(2\pi s_2r_2)dr_2ds_1ds_2\notag \\
=\int_0^{t_2}\int_0^{t_1}\phi(s_1,s_2)s_1s_2ds_1ds_2.
\end{align}

By applying $\frac{d}{du}(u^\nu J_\nu(u))=u^\nu J_{\nu-1}(u)$, and differentiating both sides of \eqref{abc} with respect to $t_1$ and $t_2$, we obtain
\begin{align*}
&\int_0^\infty\int_0^\infty\phi(s_1,s_2)s_1s_2\int_0^\infty\int_0^\infty (2\pi r_1t_1)J_0(2\pi t_1r_1)J_0(2\pi s_1r_1)dr_1 \\
&\qquad\qquad (2\pi r_2t_2)J_0(2\pi t_2r_2)J_0(2\pi s_2r_2)dr_2ds_1ds_2 \\
&=\phi(t_1,t_2)t_1t_2.
\end{align*}
which  proves \eqref{TBP}.

\smallskip

\noindent {\bf Step 3.}
Using the results of the Step 1 and 2, there exists a function $f$ on $\mathbf{R}^2$ such that
\begin{align*}
\mathcal{F}_{2,2}(\phi)(r_1,r_2)&=(2\pi)^2\int_0^\infty\int_0^\infty\phi(s_1,s_2)J_0(2\pi s_1r_1)s_1ds_1J_0(2\pi s_2r_2)s_2ds_2 \notag\\
&=(2\pi)^2\int_0^\infty\int_0^\infty\int_0^\pi\int_0^\pi f(u_1,u_2)J_0(2\pi u_1sx_1)J_0(2\pi u_2s_2)u_1u_2du_1du_2\notag\\
&\qquad \qquad J_0(2\pi s_1r_1)s_1ds_1J_0(2\pi s_2r_2)s_2ds_2 \notag\\
&=f(r_1,r_2)\chi_{[-\pi,\pi]\times[-\pi,\pi]}(r_1,r_2)\notag\\
&=4\int_{r_2}^{\pi}\int_{r_1}^{\pi}\frac{\partial^2\widehat{\phi}}{\partial w_2\partial w_1}(w_1,w_2)
\frac{dw_1}{\sqrt{w_1^2-r_1^2}}\frac{dw_2}{\sqrt{w_2^2-r_2^2}}\chi_{[0,\pi]\times[0,\pi]}(r_1,r_2)
\end{align*}
which proves \eqref{xcvfgtr} when $m=2$.

Applying \eqref{88} with $m=2,n=2$, we obtain
\begin{align*}
\mathcal{F}_{4,4}(\phi)(r_1,r_2)
&=\Big(-\frac{1}{2\pi}\frac{1}{r_2}\Big)\Big(-\frac{1}{2\pi}\frac{1}{r_1}\Big)\frac{\partial^2}{\partial r_2\partial r_1}
\left\{\mathcal{F}_{2,2}(\phi)(r_1,r_2)\right\}\notag\\
&=4\Big(-\frac{1}{2\pi}\frac{1}{r_2}\Big)\Big(-\frac{1}{2\pi}\frac{1}{r_1}\Big)\frac{\partial^2}{\partial r_2\partial r_1}
\left\{\int_{r_2}^{\pi}\int_{r_1}^{\pi}\frac{\partial^2\widehat{\phi} }{\partial w_2\partial w_1}
\frac{dw_1}{\sqrt{w_1^2-r_1^2}}\frac{dw_2}{\sqrt{w_2^2-r_2^2}}\right\}\notag\\
&=4\Big(-\frac{1}{2\pi}\frac{1}{r_2}\Big)\Big(-\frac{1}{2\pi}\frac{1}{r_1}\Big)\frac{\partial^2}{\partial r_2\partial r_1}\notag\\
&\qquad\left\{\int_{r_2}^{\pi}\int_{r_1}^{\pi}\frac{\partial}{\partial w_2}\left(\frac{1}{w_2}\frac{\partial}{\partial w_1}\left(\frac{1}{w_1}\frac{\partial^2\widehat{\phi} }{\partial w_2\partial w_1}\right)\right)
\sqrt{w_1^2-r_1^2}dw_1\sqrt{w_2^2-r_2^2}dw_2\right\}\notag\\
&=4\frac{1}{(2\pi)^2}\int_{r_2}^{\pi}\int_{r_1}^{\pi}\frac{\partial}{\partial w_2}\left(\frac{1}{w_2}\frac{\partial}{\partial w_1}\left(\frac{1}{w_1}\frac{\partial^2 \widehat{\phi} }{\partial w_2\partial w_1}\right)\right)
\frac{dw_1}{\sqrt{w_1^2-r_1^2}}\frac{dw_2}{\sqrt{w_2^2-r_2^2}},
\end{align*}
where $(r_1,r_2)\in(0,\pi)\times(0,\pi).$

Iterating this procedure,  we complete the proof when $m=2$. The case of general $m$ presents only notational differences and can
be easily deduced  by induction.
\end{proof}

\section{Applications and Examples}

\subsection{Applications to bilinear Marcinkiewicz operators} \label{subsec:m}

Let us first recall the setting of bilinear Fourier multipliers. On $\R^n$, a bilinear operator $T$ acting from $\s(\R^n) \times \s(\R^n)$ into $\s'(R^n)$ is a bilinear Fourier multiplier if it commutes with the simultaneous translations. Equivalently, there exist a bilinear kernel $K \in \s'(\R^{2n})$ and a bilinear symbol $m\in \s'(\R^{2n})$ such that for every smooth functions $f,g,h\in \s(\R^n)$ we have the two following representations:
\begin{align*}
\langle T(f,g), h\rangle & = \int_{\R^{3n}} K(y,z) f(x-y) g(x-z) h(x) \, dx\, dy\, dz \\
 &  = \int_{\R^{2n}} m(\xi,\eta) \widehat{f}(\xi) \widehat{g}(\eta) \widehat{h}(\xi+\eta) \, d\xi\, d\eta.
 \end{align*}
The kernel $K$ and the symbol $m$ are related by the Fourier transform $ K = \widehat{m}$.
We denote by $T_K$ the bilinear operator associated to the kernel $K$.

Then consider a bi-even bilinear kernel $K$ on $\R^2$ and exponents $p_1,p_2\geq 1$ such that the bilinear operator $T_K$ is bounded from $L^{p_1}(\R) \times L^{p_2}(\R)$ into $L^p(\R)$, where $p$ is given by the H\"older scaling $p^{-1}=p_1^{-1}+ p_2^{-1}$.
Now for $n\geq 2$, we may consider the bilinear kernel defined on $\R^n$ by
$$ \widetilde{K}(y,z) = (|z| |y|)^{-(n-1)} K(|y|,|z|),$$
where the factor $(|z| |y|)^{-(n-1)}$ is implicitly { dictated} by the H\"older scaling.
A natural question arises: which assumptions allow us to transport the ($L^{p_1}(\R) \times L^{p_2}(\R)\to L^p(\R)$)-boundedness of $T_K$ to a ($L^{p_1}(\R^n) \times L^{p_2}(\R^n)\to L^p(\R^n)$)-boundedness of $T_{\widetilde{K}}$ ? \\
That would correspond to the bilinear version of  results in \cite{CW}, where such a question is studied in the linear setting.

\bigskip

To answer such a question, it could be first interesting to see how this transformation $K\to \widetilde{K}$ acts on different classes of bilinear operators which are known to be bounded, such { as bilinear Calder\'on-Zygmund operators, and bilinear
multiplier operators whose symbols satisfy the H\"ormander or the Marcinkiewicz condition.}
It is obvious that the Calder\'on-Zygmund conditions on the kernel are not preserved by the transformation $K\to \widetilde{K}$.

Using the previous results, we can begin to give a positive answer in the setting of bilinear Marcinkiewicz operators. Let us first recall that a bilinear Fourier multiplier $T_K$ is called of  Marcinkiewicz type if its bilinear symbol $m$ satisfies the following regularity condition:
\begin{equation}
\sup_{\xi,\eta} \ |\xi|^{|\alpha|} |\eta|^{|\beta|} \left| \partial_\xi^\alpha \partial_\eta^\beta m(\xi,\eta) \right| \leq C_{\alpha,\beta},
\label{eq:mar}
\end{equation}
for every multi-indices $\alpha,\beta$.

\medskip

Then we have the following:
\begin{proposition} If $T_K$ is a bilinear Fourier multiplier on $\R$ of  Marcinkiewicz type then for every odd dimension $n\geq 3$, the bilinear operator $T_{\widetilde K}$ is also a bilinear Fourier multiplier of  Marcinkiewicz type on $\R^n$.
\end{proposition}

\begin{proof} Let $\tilde{m}$ the bilinear symbol associated to $\widetilde{K}$. So
$$ \widetilde{m}(\xi,\eta) = \widehat{\widetilde{K}}(\xi,\eta)=\mathcal{F}_{n,n}((r_1r_2)^{-(n-1)}K)(|\xi|,|\eta|),$$
and we have (since $K$ is assumed to be multi-even)
$$ \mathcal{F}_{1,1}((r_1r_2)^{-(n-1)}K)(r_1,r_2) = M^n(r_1,r_2),$$
where $M^n$ is the $(n-1)^{th}$-primitive of the symbol $m$ (on each coordinate) given by
$$ M^n(r_1,r_2) = \left(\int_0^{r_1}\int_0^{t_{n-1}} \cdots \int_0^{t_{2}}\right) \left(\int_0^{r_2}\int_0^{s_{n-1}} \cdots \int_0^{s_2}\right) m(t_1,s_1) \, dt_{1} ... dt_{n-1} ds_{1} ... ds_{n-1}.$$
Applying Theorem 1.1, it comes that since $m$ satisfies the regularity property \eqref{eq:mar} in $\R$, then $\widetilde{m}$ satisfies the same in $\R^n$. \\
Indeed, Theorem 1.1 yields that $\widetilde{m}$ is a sum of terms of the form
$$ \frac{1}{|\xi|^{2k-\ell_1} |\eta|^{2k-\ell_2}} \frac{\partial^{\ell_1+\ell_2}}{\partial r_1^{\ell_1}\partial r_2^{\ell_2}} M^n(|\xi|,|\eta|).$$
However the regularity on $m$ implies the following estimates on $M^n$
$$
\sup_{r_1,r_2} \ r_1^{\alpha-(n-1)} r_2^{\beta-(n-1)} \left| \partial_{r_1}^\alpha \partial_{r_2}^\beta M^n(r_1,r_2) \right| \lesssim C_{\alpha,\beta},
$$
hence we deduce that $\widetilde{m}$ is of Marcinkiewicz type on $\R^n$.
\end{proof}

We refer the reader to \cite{GK} by the second author and Kalton, where they studied the boundedness of bilinear Marcinkiewicz-type Fourier multipliers. More precisely in \cite[Theorem 7.3]{GK}, a criterion is found to be almost equivalent to the boundedness from $L^{p_1} \times L^{p_2}$ into $L^p$ and it is surprising to see that this criterion does not depend on $p_1,p_2,p$.
It could be interesting to develop this approach and study if this criterion is preserved by our transformation $K \to \widetilde{K}$.

We also refer the reader to \cite{CW} where a similar result was proved in the linear case via a similar idea. A minor difference is that the following companion recurrence formula in \cite{G1} on page 425
$$
\frac{d}{d t}(t^{\nu}J_\nu(t))=t^\nu J_{\nu-1}(t)
$$
was used in the proof of \cite[Theorem 1.8]{CW}, which results in a recursion formula which is decreasing in the dimension.

\subsection{Examples}

The following facts are known; see for instance
Appendix C in \cite{S}.
For $a,b>0$ and $x,\xi\in\mathbf{R}^1$, the Fourier transform of
$$
f(x) = \begin{cases}
\dfrac{\cos(b\sqrt{a^2-x^2}\, )}{\sqrt{a^2-x^2}} & \qquad \textup{if $|x|<a$} \\
0 & \qquad \textup{if $|x|>a$}
\end{cases}
$$
is the function $\xi\mapsto \pi J_0(a\sqrt{b^2+4\pi^2\xi^2})$
and
the Fourier transform of
$$
g(x) = \begin{cases}
\dfrac{\cosh(b\sqrt{a^2-x^2}\, )}{\sqrt{a^2-x^2}} & \qquad \textup{if $|x|<a$} \\
0 & \qquad \textup{if $|x|>a$}
\end{cases}
$$
is
\begin{equation}\label{SnF2}
G(\xi) = \begin{cases}
\pi J_0(a\sqrt{4\pi^2\xi^2-b^2}) & \qquad \textup{if $2\pi|\xi|>b$} \\
\pi J_0(ai\sqrt{b^2-4\pi^2\xi^2}) & \qquad \textup{if $2\pi|\xi|<b$}.
\end{cases}
\end{equation}

Another useful formula is that    if $h(x)=\dfrac{\sin(b\sqrt{a^2+x^2}\, )}{\sqrt{a^2+x^2}}$, then
\begin{equation}\label{SnF3}
\widehat{h}(\xi) = \begin{cases}
\pi J_0(a\sqrt{b^2-4\pi^2\xi^2}) & \qquad \textup{if $|2\pi\xi|<b$} \\
0 & \qquad \textup{if $|2\pi\xi|>b$.}
\end{cases}
\end{equation}

We have the following examples:

\noindent \textbf{Example 1.}
On $\mathbf R^{2n}$ consider the function
$$
\Phi(x,y)=
\dfrac{\cos(\sqrt{4\pi^2-|x|^2}\sqrt{4\pi^2+|y|^2})}{\sqrt{4\pi^2- |x|^2}}\chi_{(0,2\pi)}(|x|)\chi_{(0,2\pi)}(|y|)
$$
Clearly $\Phi(x,y)=\phi(|x|,|y|)$ for some function $\phi$ on $\mathbf R^2$.
Obviously, $\Phi\in L^1(\mathbf{R}^{2n})$ for all $n\geq1.$ 

{    First, we fix $y\in\mathbf R^1$, and then using the first formula of the preceding facts we calculate that the Fourier transform of $\Phi$ associated with the first variable on $\mathbf R^1$ is 
$$
\widehat{\Phi_y}(\xi,y)=\pi J_0(2\pi\sqrt{4\pi^2+y^2+4\pi^2\xi^2})\chi_{(0,2\pi)}(|y|).$$

Second, applying the inverse version of the first formula and the convolution theorem of Fourier transforms, we get that the Fourier transform of $\Phi$ on $\mathbf R^2$ is}
$$
{F}_{1,1}(\Phi)(\xi,\eta)=\left\{\frac{\cos(4\pi^2\sqrt{1+|\xi|^2}\sqrt{1-|\cdot|^2}\,)}{\sqrt{1-|\cdot|^2}}\chi_{(0,1)}(|\cdot|)\right\}*\Bigg\{\dfrac{1}{|\cdot|}\sin(4\pi^2|\cdot|)\Bigg\}(\eta),
$$
where the convolution is in the one-dimensional dotted variable.
By an easy change of variables, we rewrite the preceding formula as
$$
\mathcal{F}_{1,1}(\phi)(r_1,r_2)=\left\{\frac{\cos(4\pi^2\sqrt{1+r_1^2}\sqrt{1-|\cdot|^2}\,)}{\sqrt{1-|\cdot|^2}}\chi_{(0,1)}(|\cdot|)\right\}*\Bigg\{\dfrac{1}{|\cdot|}\sin(4\pi^2|\cdot|)\Bigg\}(r_2),
$$
where $|\xi|=r_1$ and $|\eta|=r_2.$

{Note that
$$
(-\frac{1}{2\pi r_2}\frac{\partial}{\partial r_2})(-\frac{1}{2\pi r_1}\frac{\partial}{\partial r_1})
\left[\frac{\cos(4\pi^2\sqrt{1+r_1^2}\sqrt{1-r_2^2}\,)}{\sqrt{1-r_2^2}}\right]=\frac{4\pi^2\cos(4\pi^2\sqrt{1+r_1^2}\sqrt{1-r_2^2}\,)}{\sqrt{1-r_2^2}}.
$$
 
Finally} using \eqref{88} with $m=2, n=1$, after an algebraic manipulation  and in view of the identity $\frac{d}{dr}(f\ast g)(r)=( \frac{df}{dr}  \ast g)(r)$,  we obtain that on $\mathbf{R}^{3\times3}$ we have
$$
F_{3,3}(\Phi)(\xi,\eta)=\left\{\frac{4\pi^2\cos(4\pi^2\sqrt{1+|\xi|^2}\sqrt{1-|\cdot|^2}\,)}{\sqrt{1-|\cdot|^2}}\chi_{(0,1)}(|\cdot|)\,\textup{sgn}(\cdot)\right\}*\Bigg\{\dfrac{1}{|\cdot|}\sin(4\pi^2|\cdot|)\Bigg\}(|\eta|),
$$
where $\xi\in\mathbf{R}^3,\eta\in\mathbf{R}^3$ and the convolution is in the one-dimensional dotted variable.

Next we have an example in the case $n_1\neq n_2$.

\noindent  \textbf{Example 2.} For $x\in \mathbf R^2$ and $y\in \mathbf R$ set
$$
\Phi(x,y) = \begin{cases} \dfrac{\cosh(\sqrt{4\pi^2-| x|^2}\sqrt{4\pi^2-y^2})}{\sqrt{4\pi^2- | x|^2} } & \textup{when $| x|<2\pi$, $| y|<2\pi$,} \\
0&\qquad \textup{otherwise.}
\end{cases}
$$
Obviously, $\Phi\in L^1(\mathbf{R}^{n})$ for all $n\geq3 $  and $\Phi(x,y)$ has
the form $\phi(|x|,|y|)$ for some function $\phi$ on $\mathbf R^2$.

{ By the same argument as in Example 1, indeed making use of \eqref{SnF2}, \eqref{SnF3} and the inverse version of \eqref{SnF3} respectively, we obtain}
$$
\mathcal{F}_{2,1}(\phi)(r_1,r_2)=2\pi^2 \Big(J_0\Big(4\pi^2\sqrt{r_1^2-1}\sqrt{1-|\cdot|^2}\Big)\chi_{(0,1)}(|\cdot|)\Big)*\Big(\dfrac{1}{|\cdot|}\sin(4\pi^2|\cdot|)\Big)(r_2).
$$

Applying the identity $\frac{d}{dr}J_0(r)=-J_1(r),$ $\frac{d}{dr}J_1(r)=r^{-1}J_1(r)-J_2(r)$ from B.2 (1) in \cite{G1}, it follows from a small modification of  \eqref{88}
that   $F_{4,3}(\Phi)(\xi,\eta)$ is equal to
\begin{align*}
& \left\{\!\left(\!\frac{4\pi^2 J_1(4\pi^2\sqrt{|\xi|^2-1}\sqrt{1-|\cdot|^2}\,)}{\sqrt{|\xi|^2-1}\sqrt{1-|\cdot|^2}}-8\pi^4\sqrt{|\xi|^2-1}\, J_2(4\pi^2\sqrt{|\xi|^2-1}\sqrt{1-|\cdot|^2} )\!\right)\!  \chi_{(0,1)}(|\cdot|)\, \textup{sgn}(\cdot)\right\}\notag\\
&* \Bigg\{\dfrac{1}{|\cdot|}\sin(4\pi^2|\cdot|)\Bigg\}(|\eta|),
\end{align*}
on $\mathbf{R}^{4\times3}$ where $\xi\in\mathbf{R}^4,\eta\in\mathbf{R}^3.$ Again the convolution is one-dimensional.

The following example shows how to obtain the two-dimensional Fourier transform of a radial function whose corresponding
one-dimensional Fourier transform is
compactly supported.

\noindent \textbf{Example 3.} For $t\in\mathbf{R},$ consider the even function
$$
\phi(t) = \frac{\sin(2\pi\sqrt{1+t^2}\, )}{\sqrt{1+t^2}}
$$
and define a   square-integrable   function   on $\mathbf{R}^2$ by setting $\Phi(x) = \phi(|x|) $.
Applying \eqref{SnF3} we obtain
$$
\widehat{\phi}(\tau) = \pi J_0 \big(2\pi \sqrt{1-|\tau|^2} \, \big) \chi_{|\tau|<1}
$$
for $\tau\in \mathbf R$. Then we apply  \eqref{TBSoo99} to deduce that for $r\in [0,1)$ we have
$$
\mathcal{F}_2(\phi)(r)=  2\pi \int_r^1 \frac{d}{dt} J_0 \big(2\pi \sqrt{1-t^2}\, \big) \frac{dt}{\sqrt{t^2-r^2}} =
(2\pi)^2 \int_r^1   J_1 \big(2\pi \sqrt{1-t^2}\, \big)  \frac{ t}{\sqrt{1-t^2 }} \frac{dt}{\sqrt{t^2-r^2}}\, ,
$$
where the last identity is due to the fact that $J_0'=J_{-1} = -J_1$.  Setting $u=\sqrt{1-t^2}$ we rewrite the preceding integral as
$$
\mathcal{F}_2(\phi)(r) =
(2\pi)^2\int_0^{\sqrt{1-r^2}}   J_1 \big(2\pi   u \big)   \frac{du}{\sqrt{1-r^2-u^2}}
=
-(2\pi)^2   \int_0^{1}   J_{-1} \big(2\pi \sqrt{1-r^2}\, t \big)   \frac{dt}{\sqrt{1-t^2}} \, .
$$
Using the identity B.3   in \cite{G1} (with $\mu = -1$, $\nu=-1/2$\footnote{The identity is only stated for $\mu >-1/2$ but it is also valid for $\mu>-3/2$ by analytic continuation.}) the preceding expression is equal to
$$
 \Gamma(1/2) 2^{-1/2} \frac{J_{-1/2}\big(2\pi \sqrt{1-r^2}\,   \big) }{ \big(2\pi \sqrt{1-r^2}\,   \big)^{1/2}} =
  \frac{\cos \big(2\pi \sqrt{1-r^2}\,   \big) }{ 2\pi \sqrt{1-r^2} }\, .
$$
This  provides a formula for the two-dimensional Fourier transform $\widehat{\Phi}$ of $\Phi$ as a function of $r=|\xi|$ when $r\in [0,1)$. Notice that $\widehat{\Phi}(\xi)$ vanishes when $|\xi|\ge 1$.

\end{document}